\documentclass[reqno,a4paper]{article}

\usepackage{rotating}
\usepackage{amsmath}
\usepackage{amsfonts}
\usepackage{amssymb}
\usepackage{amscd}
\usepackage{amsthm}
\usepackage{stmaryrd}
\usepackage[all]{xy}
\xyoption{line}
\SelectTips{cm}{10}
\usepackage{graphics}
\usepackage{graphicx}
\usepackage{fancyhdr}
\usepackage{enumitem}
\usepackage{mathtools}
\usepackage{url}

\theoremstyle{plain}
\newtheorem{theorem}{Theorem}[section]

\theoremstyle{definition}

\theoremstyle{remark}
\newtheorem*{remark}{Remark}

\newtheoremstyle{tiny}
     {\topsep}
     {\topsep}
     {\footnotesize}
     {}
     {\itshape}
     {.}
     {.5em}
     {}

\theoremstyle{tiny}
\newtheorem*{tremark}{Remark}

\newcommand{\comma}{{},\,}
\newcommand{\col}{\!:\:\!}
\newcommand{\loc}{\!\::\!}

\newcommand{\bbI}{I\!\!I}
\newcommand{\bbZ}{\mathbb Z}
\newcommand{\mcC}{\mathcal C}
\newcommand{\mcD}{\mathcal D}
\newcommand{\mcI}{\mathcal I}
\newcommand{\mcM}{\mathcal M}
\newcommand{\mcO}{\mathcal O}
\newcommand{\mcS}{\mathcal S}

\newcommand{\bbone}{\bbI^0}

\newcommand{\sSet}{\mathsf{sSet}}

\newcommand{\ra}{\rightarrow}
\newcommand{\la}{\leftarrow}
\newcommand{\xlra}[1]{\xrightarrow{\ #1\ }}
\newcommand{\xlla}[1]{\xleftarrow{\ #1\ }}
\newcommand{\lra}{\longrightarrow}

\newdir{c}{{}*!/-5pt/@^{(}}
\newdir{d}{{}*!/-5pt/@_{(}}
\newdir{ >}{{}*!/-5pt/@{>}}
\newdir{s}{{}*!/+10pt/@{}}
\newdir{|>}{{}*!/2pt/@{|}*@{>}}

\newcommand{\cof}[1][]{\mathbin{\:\!\!\xymatrix@1@C=15pt{{}\ar@{ >->}[r]^{#1} & {}}}}
\newcommand{\fib}[1][]{\mathbin{\:\!\!\xymatrix@1@C=15pt{{}\ar@{->>}[r]^{#1} & {}}}}

\newcommand{\pbsize}{15pt}
\newcommand{\pboffset}{.5}

\newcommand{\xycorner}[3]{\save #2="a";#1;"a"**{}?(\pboffset);"a"**\dir{-};#3;"a"**{}?(\pboffset);"a"**\dir{-}\restore}
\newcommand{\pb}{\xycorner{[]+<\pbsize,0pt>}{[]+<\pbsize,-\pbsize>}{[]+<0pt,-\pbsize>}}
\newcommand{\po}{\xycorner{[]+<-\pbsize,0pt>}{[]+<-\pbsize,\pbsize>}{[]+<0pt,\pbsize>}}

\newcommand{\xyhcorner}[3]{\save #2="a";#1;"a"**{}?(\pboffset);"a"**\dir{-};#3;"a"**{}?(\pboffset);"a"**\dir{-};[]**{}?(.3)*{\scriptstyle h} \restore}

\newcommand{\hpo}{\xyhcorner{[]+<-\pbsize,0pt>}{[]+<-\pbsize,\pbsize>}{[]+<0pt,\pbsize>}}

\newcommand{\xymatrixc}[2][]{\xy *!C\xybox{\xymatrix#1{#2}}\endxy}

\newcommand{\leftbox}[2]{\phantom{#1} \save []+L*+<.5pc>!!<0pt,\the\fontdimen22\textfont2>!L{#1#2} \restore}
\newcommand{\rightbox}[2]{\phantom{#2} \save []+R*+<.5pc>!!<0pt,\the\fontdimen22\textfont2>!R{#1#2} \restore}

\newcommand{\cleftbox}[2]{#1#2 \POS[]+L*+<.5pc>!!<0pt,\the\fontdimen22\textfont2>!L{\phantom{#1}\vphantom{#2}}+C*+<.5pc>{\phantom{#1}\vphantom{#2}}}
\newcommand{\crightbox}[2]{#1#2 \POS[]+R*+<.5pc>!!<0pt,\the\fontdimen22\textfont2>!R{\vphantom{#1}\phantom{#2}}+C*+<.5pc>{\vphantom{#1}\phantom{#2}}}

\newcommand{\tdcong}{\rotatebox{-90}{$\cong$}}

\entrymodifiers={+!!<0pt,\the\fontdimen22\textfont2>}

\newcommand{\stdsimp}[1]{\Delta^{#1}}
\newcommand{\horn}[2]{%
\xy
<0pt,-\the\fontdimen22\textfont2>;p+<.1em,0em>:
{\ar@{-}(0,0);(3,7)},
{\ar@{-}(3,7);(5,0)},
{\ar@{-}(3.2,7);(5.2,0)},
{\ar@{-}(3.4,7);(5.4,0)}
\endxy\;\!{}^{#1}_{#2}}

\newcommand{\conn}{\operatorname{conn}}
\newcommand{\defeq}{\stackrel{\mathrm{def}}{=}}

\newcommand{\hocolim}{\operatorname{hocolim}}
\newcommand{\id}{\operatorname{id}}
\newcommand{\incl}{\operatorname{in}}
\newcommand{\map}{\mathop\mathsf{map}\nolimits}
\newcommand{\op}{\mathrm{op}}
\newcommand{\xysmallbullet}{*+<.2pc>!!<0pt,\the\fontdimen22\scriptfont2>{\scriptstyle\bullet}}

\newcommand{\thed}{d}
\newcommand{\thek}{k}
\newcommand{\then}{n}

\usepackage{titlesec}

\titleformat{\section}[block]
{\normalfont\Large\filcenter\bfseries}{\thesection.}{.33em}{}
\titlespacing*{\section}
{0pt}{3.5ex plus 1ex minus .2ex}{2.3ex plus .2ex}

\titleformat{\subsection}[runin]
{\normalfont\normalsize\bfseries}{\thesubsection.}{.33em}{}[.]
\titlespacing*{\subsection}
{0pt}{3.25ex plus 1ex minus .2ex}{.5em}

\title{Heaps and unpointed stable homotopy theory\thanks{%
The research was supported by the grant P201/11/0528 of the Czech Science Foundation (GA \v CR).\newline
2010 \emph{Mathematics Subject Classification}. Primary 55U35; Secondary 55P42.\newline
\emph{Key words and phrases}. stable homotopy, heap, equivariant, fibrewise.
}}

\author
{Luk\'a\v{s} Vok\v{r}\'{\i}nek}

\begin{document}

\maketitle

\begin{abstract}
In this paper, we show how certain ``stability phenomena'' in \emph{unpointed} model categories provide the sets of homotopy classes with the structure of abelian heaps, i.e.\ abelian groups without a choice of a zero. In contrast with the classical situation of stable (pointed) model categories, these sets can be empty.
\end{abstract}

\section{Introduction}

This paper grew out of an attempt to understand the appearance of non-canonical abelian group structures on sets of equivariant fibrewise homotopy classes of maps under certain stability restrictions (dimension vs.\ connectivity), utilized in \cite{aslep} for the algorithmic computation of this set. Classically, the abelian group structures on stable homotopy classes of maps rely on the presence of basepoints -- in particular, the constant map onto the basepoint serves as the zero element of this group.

However, there are situations in which basepoints do not exist, e.g.\ for spaces equipped with a free action of a fixed group $G$ or spaces over a fixed base space $B$ that do not admit any section. In such situations, there is no canonical choice of a zero element and as a result, the structure on the set of homotopy classes turns out to be more naturally an abelian heap; moreover, it could be empty. This structure is constructed in \cite{heaps2} from the Moore--Postnikov tower of the target and as such is bound to the specific situation of that paper. The present paper explains the situation more conceptually.

Now we will state the main result of this paper. It uses the notions of $d$-connected and $n$-dimensional objects that will only be explained later but which in many cases (most notably those mentioned above) have a straightforward interpretation. Heaps are defined formally after the statement -- they are essentially groups without a choice of a zero.

\begin{theorem}\label{t:heap_structure_abstract}
Let $\mcM$ be a simplicial model category. Then the set $[X,Y]$ of homotopy classes of maps from an $n$-dimensional cofibrant object $X$ to a $d$-connected fibrant object $Y$ with $n\leq 2d$ admits a canonical structure of a (possibly empty) abelian heap.
\end{theorem}

In the last section, we outline a construction of a category of finite spectra in the spirit of Spanier and Whitehead.

\subsection*{Heaps}

A \emph{Mal'cev operation} on a set $S$ is a ternary operation
\[t\colon S\times S\times S\to S\]
satisfying the following two conditions: $t(x,x,y)=y$, $t(x,y,y)=x$. It is said to be
\begin{itemize}[topsep=2pt,itemsep=2pt,parsep=2pt,leftmargin=\parindent,label=--]
\item
	\emph{asssociative} if $t(x,r,t(y,s,z))=t(t(x,r,y),s,z)$;
\item
	\emph{commutative} if $t(x,r,y)=t(y,r,x)$.
\end{itemize}
A set equipped with an associative Mal'cev operation is called a \emph{heap}. It is said to be an \emph{abelian heap} if in addition, the operation is commutative. We remark that traditionally, heaps are assumed to be non-empty. Since it is easy to produce examples where $[X,Y]=\emptyset$ in Theorem~\ref{t:heap_structure_abstract}, it will be more convenient to drop this convention.

The relation of heaps and groups works as follows. Every group becomes a heap if the Mal'cev operation is defined as $t(x,r,y)=x-r+y$. On the other hand, by fixing an element $0\in S$ of a heap $S$, we may define the addition and the inverse
\[x+y=t(x,0,y),\quad -x=t(0,x,0).\]
It is simple to verify that this makes $S$ into a group with neutral element $0$. In both passages, commutativity of heaps corresponds exactly to the commutativity of groups.

\subsection*{Acknowledgement}
We are grateful to Martin \v{C}adek for carefully reading a draft of this paper.

\section{Suspensions and loop spaces in unpointed model categories}

We work in a simplicial model category $\mcM$ with the enriched hom-set denoted by $\map(X,Y)$, tensor by $K\otimes X$ and cotensor by $Y^K$. Examples that we have in mind are $G$-spaces over a fixed $G$-space $B$, or diagrams of such, see Section~\ref{s:examples}. We denote by $I$ the simplicial set $\xymatrix@1@C=1pc{\xysmallbullet & \xysmallbullet \ar[l] \ar[r] & \xysmallbullet}$ formed by two standard $1$-simplices glued along their initial vertices, and by $\partial I$ its obvious ``boundary'' composed of the two terminal vertices. Further, we denote by $\bbone$ the cofibrant fibrant replacement of the terminal object. The standard references for simplicial sets, model categories and homotopy colimits are \cite{GoerssJardine,Hirschhorn}.

We define a Quillen adjunction $\Sigma\dashv\Omega$ composed of the \emph{suspension} and \emph{loop space} functors%
\footnote{
	One may also use $\mcM/\bbone\times\bbone$ instead of $\mcM/\bbone$ (later replacing equalizer--cokernel pair by pullback--pushout), thus producing a more symmetric adjunction. However, the non-symmetric version is easier to generalize to higher suspensions -- these will be needed later.
}
\[\xymatrix{
\Sigma\col\mcM/\bbone \ar@<2pt>[r] & \bbone\sqcup\bbone/\mcM\loc\Omega \ar@<2pt>[l]
}\]
where $\mcM/\bbone$ is the slice category of objects over $\bbone$ while $\bbone\sqcup\bbone/\mcM$ is the slice category of objects under $\bbone\sqcup\bbone$. The functor $\Sigma$ is defined on $p\colon X\to\bbone$ as a homotopy pushout
\[\xymatrix@C=3pc{
X \ar[r]^-p \ar[d]_-p & \bbone \ar[d]^-j \\
\bbone \ar[r]_-i & \Sigma X \hpo
}\]
with the two components $i\comma j$ of the universal cone making $\Sigma X$ into an object under $\bbone\sqcup\bbone$. In other words, the maps $i$, $j$ form the homotopy cokernel pair of $p$. The loop space functor is defined on $[i,j]\colon\bbone\sqcup\bbone\lra Y$ as the homotopy equalizer of $i$ and $j$. In the presence of a simplicial enrichment, there are standard models for homotopy (co)limits that translate into the following pushout/pullback squares:
\[\xymatrix@C=3pc{
\partial I\otimes X \ar[r]^-{\id\otimes p} \ar[d]_-{\mathrm{incl}} &
	\leftbox{\partial I\otimes\bbone}{{}\cong\bbone\sqcup\bbone} \ar[d]^-{[i,j]} & &
	\Omega Y \ar[r] \ar[d]_-p \pb &
	Y^I \ar[d]^-{\mathrm{res}} \\
I\otimes X \ar[r] &
	\Sigma X \po & &
	\bbone \ar[r]_-{(i,j)} &
	Y^{\partial I}
}\]
From this restatement, it follows rather easily that $\Sigma\dashv\Omega$ is a Quillen adjunction.

Since the unique map $\bbone\to 1$ to the terminal object is a weak equivalence between fibrant objects, it is easy to see that there is a Quillen equivalence $\mcM/\bbone\simeq_Q\mcM$. Thus, one may think of $\Sigma$ as being defined on $\mcM$ while $\Omega$ is defined on objects equipped with a ``pair of basepoints'' (and $\Omega Y$ is then the space of paths from the first basepoint to the second).

For a cofibrant object $Y$, we consider the \emph{derived unit} $\eta_Y\colon Y\to\Omega(\Sigma Y)^\mathrm{fib}$, where the superscript ``$\mathrm{fib}$'' denotes the fibrant replacement of $\Sigma Y$. To state an abstract version of a Freudenthal suspension theorem, we need a notion of a $\thed$-equivalence.

\subsection*{Abstract theorems of Freudenthal and Whitehead}

\newcommand{\excisive}{excisive}

We say that a cofibrant object $D\in\mcM$ is \emph{\excisive} if the right derived functor of $\map(D,-)$ preserves homotopy pushouts in the following sense: when $Y\colon\mcS\to\mcM$ is a diagram consisting of fibrant objects, indexed by the span category $\mcS=\xymatrix@1@C=1pc{\xysmallbullet & \xysmallbullet \ar[l] \ar[r] & \xysmallbullet}$, then the composition
\[\hocolim_\mcS\map(D,Y{-})\lra\map(D,\hocolim_\mcS Y)\lra\map(D,(\hocolim_\mcS Y)^\mathrm{fib})\]
is a weak equivalence.

Let us fix a collection $\mcD\subseteq\mcM^\mathrm{cof}$ of cofibrant \excisive{} objects.
\begin{itemize}[topsep=2pt,itemsep=2pt,parsep=2pt,leftmargin=\parindent,label=$\bullet$]
\item
	We say that $Y$ is \emph{$\thed$-connected} if for each $D\in\mcD$, $\map(D,Y^\mathrm{fib})$ is $\thed$-connected.
\item
	We say that a map $f\colon Y\to Z$ is a \emph{$\thed$-equivalence} if for each $D\in\mcD$, the map $f_*\colon\map(D,Y^\mathrm{fib})\to\map(D,Z^\mathrm{fib})$ is a $\thed$-equivalence of simplicial sets.
\end{itemize}

\begin{theorem}[Freudenthal]\label{t:Freudenthal}
Let $Y$ be a $\thed$-connected cofibrant object. Then the canonical map $\eta_Y\colon Y\to\Omega(\Sigma Y)^\mathrm{fib}$ is a $(2\thed+1)$-equivalence.
\end{theorem}

Let $\mcI_\then$ denote the following collection of maps:
\[\mcI_\then=\{\partial\stdsimp\thek\otimes D\to\stdsimp\thek\otimes D\mid\thek\leq\then\comma D\in\mcD\}.\]
We say that $X$ has \emph{dimension} at most $\then$ if the unique map $0\to X$ from the initial object is an $\mcI_\then$-cell complex (i.e.\ it is obtained from $\mcI_\then$ by pushouts and transfinite compositions); we write $\dim X\leq\then$. More generally, if $A\to X$ is an $\mcI_\then$-cell complex, we write $\dim_AX\leq\then$.

\begin{tremark}
It is also possible to add to $\mcI_\then$ all trivial cofibrations -- this will not change the homotopy theoretic nature of $\mcI_\then$-cell complexes.
\end{tremark}

\begin{theorem}[Whitehead]\label{t:Whitehead}
Let $X\comma Y\comma Z$ be objects of $\mcM$ and $f\colon Y\to Z$ a $\thed$-equivalence. If $\dim X\leq\thed$, then the induced map
\[f_*\colon[X,Y]\to[X,Z]\]
is surjective. If $\dim X<\thed$, the induced map is a bijection.
\end{theorem}

The usefullness of the above theorems is limited by the existence of a class $\mcD$ of \excisive{} objects for which the resulting notions of connectivity and dimension are interesting. Examples of such classes are provided in Section~\ref{s:examples}. We continue with the proof of Theorem~\ref{t:heap_structure_abstract} assuming Theorems~\ref{t:Freudenthal} and~\ref{t:Whitehead} -- these are proved in Section~\ref{s:proofs}.

\subsection*{Proof of Theorem~\ref{t:heap_structure_abstract}}\label{s:proof_heap_structure_abstract}
It follows from the Quillen adjunction $\Sigma\dashv\Omega$, Theorems~\ref{t:Freudenthal} and~\ref{t:Whitehead} that for $\dim X\leq 2\conn Y$, we have
\[[X,Y]\cong[X,Y]_{\bbone}\cong[X,\Omega(\Sigma Y)^\mathrm{fib}]_{\bbone}\cong[\Sigma X,\Sigma Y]^{\partial\bbI},\]
where we denote $\partial\bbI=\partial I\otimes\bbone\cong\bbone\sqcup\bbone$. It is rather straightforward to equip $\Sigma X\in\partial\bbI/\mcM$ with a ``weak co-Malcev cooperation'' -- this comes from such a structure on $I\in\partial I/\sSet$ given by the zig-zag
\begin{equation}\label{eq:coMalcev}
\xymatrix@!C=0pt@C=1.5pc@R=2pc{
I & & & &
	\widetilde I \ar[llll]_-\sim \ar[rrrrrr] & & & & & &
	I\sqcup_{\partial I}I\sqcup_{\partial I}I \\
\xysmallbullet \POS[]*+<.5pc>!LD{\scriptstyle\mathrm{tar}} & & & &
	\xysmallbullet & &
	\xysmallbullet \POS[]*+<.5pc>!LD{\scriptstyle\mathrm{tar}} & & & &
	\xysmallbullet \POS[]*+<.5pc>!LD{\scriptstyle\mathrm{tar}} \\
\xysmallbullet \ar@{-}[u] \POS[]*+<.5pc>!RU{\scriptstyle\mathrm{src}} & & &
	\xysmallbullet \ar@{-}[ru] \POS[]*+<.5pc>!RU{\scriptstyle\mathrm{src}} & &
	\xysmallbullet \ar@{-}[lu] \ar@{-}[ru] & & & & &
	\xysmallbullet \ar@<-2pt>@/_6pt/@{-}[u] \ar@{-}[u] \ar@<2pt>@/^6pt/@{-}[u] \POS[]*+<.5pc>!RU{\scriptstyle\mathrm{src}}
}\end{equation}
(both maps take the copies of $I$ in $\widetilde I$ onto the corresponding copies of $I$ in the target; for the second map, they are the left, the middle and the right copy). Tensor-multiplying by $X$ and collapsing the source and target copies of $X$ to $\bbone$'s, one gets
\[\Sigma X\xlla\sim\widetilde\Sigma X\lra\Sigma X\sqcup_{\partial\bbI}\Sigma X\sqcup_{\partial\bbI}\Sigma X.\]
On homotopy classes, it induces the map $t$ in the following diagram.
\[\xymatrix@R=1pc{
\crightbox{[\Sigma X,\Sigma Y]^{\partial\bbI}\times[\Sigma X,\Sigma Y]^{\partial\bbI}\times{}}{[\Sigma X,\Sigma Y]^{\partial\bbI}} \ar@/^10pt/[rrd]^-t \POS[]\ar@{}[d]|-\tdcong \\
[\Sigma X\sqcup_{\partial\bbI}\Sigma X\sqcup_{\partial\bbI}\Sigma X,\Sigma Y]^{\partial\bbI} \ar[r] & [\widetilde\Sigma X,\Sigma Y]^{\partial\bbI} & [\Sigma X,\Sigma Y]^{\partial\bbI} \ar[l]_-\cong
}\]

To prove the Mal'cev conditions, consider the homotopy commutative diagram
\[\xymatrix@C=3pc{
\widetilde\Sigma X \ar[r] \ar[d]_-\sim \POS[];[rd]**{}?<>(.5)*{\scriptstyle\mathrm{htpy}} & \Sigma X\sqcup_{\partial\bbI}\Sigma X\sqcup_{\partial\bbI}\Sigma X \ar[d]^-{[\incl_\mathrm{left},\incl_\mathrm{left},\incl_\mathrm{right}]} \\
\Sigma X \ar[r]_-{\incl_\mathrm{right}} & \Sigma X\sqcup_{\partial\bbI}\Sigma X
}\]
with the map on the right restricting to the indicated maps on the three copies of $\Sigma X$ in the domain -- they are the inclusions of $\Sigma X$ as the left or right copy in $\Sigma X\sqcup_{\partial\bbI}\Sigma X$. This easily yields the identity $t(x,x,y)=y$ and a symmetric diagram gives $t(x,y,y)=x$. Thus, $t$ is a Mal'cev operation. The associativity is equally simple to verify.

\subsection*{Higher suspensions and commutativity}

In order to get commutativity, we introduce higher suspensions. Let $\partial I^k$ be the obvious boundary of $I^k=I\times\cdots\times I$ and denote $\partial\bbI^k=\partial I^k\otimes\bbone$ and $\bbI^k=I^k\otimes\bbone$. We assume that $\mcM$ is right proper; otherwise, one has to fibrantly replace $\bbI^k$. The higher suspensions are naturally defined on $\partial\bbI^k/\mcM/\bbI^k$, the category of chains $\partial\bbI^k\xlra{i}X\xlra{p}\bbI^k$, whose composition is assumed to be the canonical inclusion. Then $\Sigma^\ell X$ is the pushout in
\[\xymatrix@C=3pc{
\partial I^\ell\otimes X \ar[r]^-{\id\otimes p} \ar[d]_-{\incl\otimes\id} &
	\partial I^\ell\otimes\bbI^k \ar[d] \ar@/^1pc/[rd]^-{\incl\otimes\id} \\
I^\ell\otimes X \ar[r] \ar@/_1.5pc/[rr]_-{\id\otimes p} &
	\Sigma^\ell X \ar@{-->}[r] \po &
	I^\ell\otimes\bbI^k
}\]
This makes $\Sigma^\ell X$ into an object over $\bbI^{\ell+k}$. The map $i$ then induces
\[\partial\bbI^{\ell+k}=\Sigma^\ell\partial\bbI^k\to\Sigma^\ell X,\]
making $\Sigma^\ell$ into a functor $\Sigma^\ell\colon\partial\bbI^k/\mcM/\bbI^k\to\partial\bbI^{k+\ell}/\mcM/\bbI^{k+\ell}$. As such, $\Sigma^\ell$ is a left Quillen functor. Moreover, it is clear that $\Sigma^{\ell_0}\Sigma^{\ell_1}\cong\Sigma^{\ell_0+\ell_1}$.

The right properness of $\mcM$ implies $\partial\bbI^k/\mcM/\bbI^k\simeq_Q\partial\bbI^k/\mcM$ and we may think of the suspensions as defined on $\partial\bbI^k/\mcM$. By an obvious generalization of Theorem~\ref{t:Freudenthal}, we obtain for $\dim X\leq 2\conn Y$ bijections
\[[X,Y]\cong[\Sigma X,\Sigma Y]^{\partial\bbI}\cong[\Sigma^2X,\Sigma^2Y]^{\partial\bbI^2}.\]

The ``square'' of \eqref{eq:coMalcev} yields the following diagram
\[\xymatrix@R=1.5pc@C=3pc{
\partial I^2 \ar[d] & \partial\widetilde I^2 \ar[l]_-h^-\sim \ar[r]^-h_-\sim \ar[d] & \partial I^2 \ar[d] \\
I^2 & \widetilde I^2 \ar[l]_-\sim \ar@<-.2pc>[r] \ar@<.2pc>[r] & \leftbox{I^2\sqcup_{\partial I^2}I^2}{{}\sqcup_{\partial I^2}I^2}
}\]
with the two parallel arrows denoting two possible ways of folding a square into three squares -- horizontally and vertically. Thus, the diagram takes place in $\partial\widetilde I^2/\sSet$. Denoting $\partial\widetilde{\bbI}^2=\partial\widetilde I^2\otimes\bbI^0$, we obtain two heap structures on $[\Sigma^2X,\Sigma^2Y]^{\partial\widetilde{\bbI}^2}$ that distribute over each other. Since the Eckman--Hilton argument holds for heaps, these structures are identical and commutative. Because $h$ is a weak equivalence, the canonical map
\[[\Sigma^2X,\Sigma^2Y]^{\partial\bbI^2}\xlra\cong[\Sigma^2X,\Sigma^2Y]^{\partial\widetilde\bbI^2}\]
is a bijection and it may be used to transport the abelian heap structure to $[\Sigma^2X,\Sigma^2Y]^{\partial\bbI^2}$.
\qed\vskip\topsep

\begin{remark}
The isomorphism in the beginning of the above proof takes the following more symmetric form for $A/\mcM$:
\[[X,Y]^A\cong[X,\Omega(\Sigma Y)^\mathrm{fib}]^A\cong[\Sigma X,\Sigma Y]^{\Sigma A}\]
(the suspension in $A/\mcM$ is weakly equivalent to that computed in $\mcM$ and since the initial object of $A/\mcM$ is $A$, we have $\partial\bbI=\Sigma 0=\Sigma A$).
\end{remark}

\section{Proofs of Freudenthal and Whitehead Theorems}\label{s:proofs}

\begin{proof}[Proof of Theorem~\ref{t:Freudenthal}]
The following diagram commutes
\[\xymatrix@R=.5pc@C=3pc{
& \Omega(\Sigma\map(D,Y))^\mathrm{fib} \ar[dd]^-\sim \\
\map(D,Y) \ar@/^10pt/[ru]^-{\eta_{\map(D,Y)}} \ar@/_10pt/[rd]_-{\eta_{Y*}} \\
& \map(D,\Omega(\Sigma Y)^\mathrm{fib})
}\]
and the vertical map is a weak equivalence since $\map(D,{-})$ commutes with homotopy limits such as $\Omega$ in general and it commutes with the homotopy pushout $\Sigma$ by our assumption of $D$ being \excisive. The map $\eta_{\map(D,Y)}$ is a $(2\thed+1)$-equivalence since the Freudenthal suspension theorem holds in simplicial sets and $\map(D,Y)$ is $\thed$-connected.
\end{proof}

\begin{proof}[Proof of Theorem~\ref{t:Whitehead}]

Denoting $\iota\colon\partial\stdsimp\thek\otimes D\to\stdsimp\thek\otimes D$, we will first show that the square
\begin{equation}\label{eq:d-equivalence_square}
\xymatrixc{
\map(\stdsimp\thek\otimes D,Y) \ar[r]^-{f_*} \ar@{->>}[d]_-{\iota^*} & \map(\stdsimp\thek\otimes D,Z) \ar@{->>}[d]^-{\iota^*} \\
\map(\partial\stdsimp\thek\otimes D,Y) \ar[r]_-{f_*} & \map(\partial\stdsimp\thek\otimes D,Z)
}\end{equation}
is $(\thed-\thek)$-cartesian, i.e.\ that the map from the top left corner to the homotopy pullback is a $(\thed-\thek)$-equivalence. Equivalently, the induced map of the homotopy fibres of the two vertical maps is a $(\thed-\thek)$-equivalence for all possible choices of basepoints.

The square \eqref{eq:d-equivalence_square} is isomorphic to
\[\xymatrix@C=3pc{
\map(D,Y)^{\stdsimp\thek} \ar[r]^-{f_*} \ar@{->>}[d]_-{\iota^*} & \map(D,Z)^{\stdsimp\thek} \ar@{->>}[d]^-{\iota^*} \\
\map(D,Y)^{\partial\stdsimp\thek} \ar[r]_-{f_*} & \map(D,Z)^{\partial\stdsimp\thek}
}\]
This square maps to the square on the left of the following diagram via evaluation at any vertex of $\stdsimp\thek$ in such a way that the corresponding homotopy fibres over $\varphi\colon D\to Y$ and $\psi=f\varphi\colon D\to Z$ are organized in the right square:
\[\xymatrix{
\map(D,Y) \ar[r]^-{f_*} \ar[d]_-{\id} & \map(D,Z) \ar[d]^-{\id} & \mathrm{contractible} \ar[r] \ar[d] & \mathrm{contractible} \ar[d] \\
\map(D,Y) \ar[r]_-{f_*} & \map(D,Z) & \Omega^{\thek-1}_{\varphi}\map(D,Y) \ar[r]_-{f_*} & \Omega^{\thek-1}_{\psi}\map(D,Z)
}\]
(the loop spaces are the usual loop spaces based at the indicated points). The square on the left is $\infty$-cartesian and in the one on the right, the map of the homotopy fibres of the vertical maps is $f_*\colon\Omega^\thek_\varphi\map(D,Y)\to\Omega^\thek_\psi\map(D,Z)$ which is indeed a $(\thed-\thek)$-equivalence.

It follows easily from the properties of $(\thed-\then)$-cartesian squares that for all $\mcI_\then$-cell complexes $\iota\colon A\to X$, the square
\[\xymatrix@C=3pc{
\map(X,Y) \ar[r]^-{f_*} \ar@{->>}[d]_-{\iota^*} & \map(X,Z) \ar@{->>}[d]^-{\iota^*} \\
\map(A,Y) \ar[r]_-{f_*} & \map(A,Z)
}\]
is also $(\thed-\then)$-cartesian. In particular, when $\then\leq\thed$ and $A=0$, we obtain a surjection on the components of the spaces at the top, i.e.\ $f_*\colon[X,Y]\to[X,Z]$ is surjective. For $\then<\thed$, it is a bijection (and the induced map on $\pi_1$ is still surjective).
\end{proof}

\section{Examples}\label{s:examples}

We will now show how to produce examples of collections of \excisive{} objects.

\begin{enumerate}[topsep=2pt,itemsep=2pt,parsep=2pt,leftmargin=\parindent,label=\alph*)]
\item\textbf{Spaces.}
In the category of simplicial sets, $\mcD=\{\stdsimp 0\}$ is a collection of \excisive{} objects. The resulting notions of $d$-equivalences and $n$-dimensional objects are the standard ones.
\end{enumerate}

\noindent In the following examples, we assume that $\mcD\subseteq\mcM$ is a collection of \excisive{} objects.

\begin{enumerate}[resume,topsep=2pt,itemsep=2pt,parsep=2pt,leftmargin=\parindent,label=\alph*)]
\item\textbf{Diagram categories.}
Let $\mcC$ be a small (simplicial) category. When $\mcM$ is cofibrantly generated, then the diagram category $\mcM^\mcC$, i.e.\ the category of (simplicial) functors $\mcC\to\mcM$, admits a projective model structure. The collection
\[\mcD'\defeq\{\mcC(c,-)\otimes D\mid c\in\mcC\comma D\in\mcD\}\]
also consists of \excisive{} objects -- this follows from the Yoneda lemma
\[\map(\mcC(c,-)\otimes D,Y)\cong\map(D,Yc)\]
and the fact that homotopy colimits in $\mcM^\mcC$ are computed pointwise.

In this way, a map $p\colon Y\to Z$ in $\mcM^\mcC$ is a $d$-equivalence if and only if each component $p_c\colon Yc\to Zc$ is a $d$-equivalence.

\item\textbf{Equivariant categories.}
Let $G$ be a group and consider the category $G{-}\mcM$ of objects equipped with a $G$-action. In the paper \cite{Stephan}, sufficient conditions on $\mcM$ are stated that provide a model structure on $G{-}\mcM$ and a Quillen equivalence $\mcM^{\mcO_G^{\phantom{G}\op}}\simeq_QG{-}\mcM$. In particular, these conditions apply to $\mcM=\sSet$. Thus, the equivariant case reduces to that of the diagram categories. The resulting collection of \excisive{} objects is
\[\mcD'\defeq\{G/H\otimes D\mid H\leq G\comma D\in\mcD\}.\]

In this way, a map $p\colon Y\to Z$ in $G{-}\mcM$ is a $d$-equivalence if and only if all fixed point maps $p^H\colon Y^H\to Z^H$ are $d$-equivalences. In $G{-}\sSet$, dimension has the usual meaning.

\item\textbf{Fibrewise categories.}
Let $B\in\mcM$ be an object. Then in the category $\mcM/B$, the collection
\[\mcD'\defeq\{f\colon D\to B\mid D\in\mcD\comma\textrm{$f$ arbitrary}\}\]
also consists of \excisive{} objects. This follows from the fact that for a fibrant object over $B$, i.e.\ a fibration $\varphi\colon Y\fib B$, there is a (homotopy) pullback square
\[\xymatrix@C=3pc{
\map_B(D,Y) \ar[r] \ar[d] \pb & \map(D,Y) \ar@{->>}[d]^-{\varphi_*} \\
{}* \ar[r]_-{f} & \map(D,B)
}\]
and in $\sSet$, homotopy pushouts are stable under homotopy pullbacks by Mather's Cube Theorem, see \cite{Mather}.

In this way, a map $p\colon Y\to Z$ in $\mcM/B$ is a $d$-equivalence if and only if it is a $d$-equi\-val\-ence in $\mcM$. Also, $\dim X\leq n$ in $\mcM/B$ if and only if the same is true in $\mcM$.

\item\textbf{Relative categories.}
Let $A\in\mcM$ be an object. Then in the category $A/\mcM$, the collection
\[\mcD'\defeq\{\incl\colon A\to A\sqcup D\mid D\in\mcD\}\]
of coproduct injections also consists of \excisive{} objects. This follows from the fact that $\map^A(A\sqcup D,Y)\cong\map(D,Y)$ and homotopy pushouts in $A/\mcM$ are computed essentially as in $\mcM$ (more precisely, the homotopy pushout of $X_1\la X_0\ra X_2$ in $A/\mcM$ is obtained from that in $\mcM$ by collapsing the copy of $I\otimes A$ to $A$; the identification map is a weak equivalence).

In this way, a map $p\colon Y\to Z$ in $A/\mcM$ is a $d$-equivalence if and only if it is a $d$-equivalence in $\mcM$. The dimension of an object $X\in A/\mcM$ equals $\dim_AX$.
\end{enumerate}

\section{The Spanier--Whitehead category of spectra}

With the unpointed suspension and loop space as a tool, we will outline a construction of a category of spectra. For simplicity and since we do not have any particular applications in mind, we will only deal with finite spectra in the spirit of Spanier and Whitehead.

As before, we assume that $\mcM$ is right proper. We say that $X\in\mcM$ is a finite complex if it is a finite $\left(\bigcup_{n\geq 0}\mcI_n\right)$-cell complex.

The objects of $\mathsf{Sp}_\mcM$ are formal (de)suspensions $\Sigma^\ell X$ -- these are simply pairs $(\ell,X)$ such that
\begin{itemize}[topsep=2pt,itemsep=2pt,parsep=2pt,leftmargin=\parindent,label=$\bullet$]
\item
	$\ell\in\bbZ$ is an arbitrary integer,
\item
	for some $k\geq-\ell$, $X\in\partial\bbI^k/\mcM^\mathrm{fin}$ is an arbitrary finite complex.
\end{itemize}
We say that $X$ is of \emph{degree} $d=\ell+k\geq 0$. If $\Sigma^{\ell_0}X_0$, $\Sigma^{\ell_1}X_1$ are two objects of the same degree $d$, we define the set $[\Sigma^{\ell_0}X_0,\Sigma^{\ell_1}X_1]$ as the colimit
\[\operatorname*{colim}_{i\geq\max\{-\ell_0,-\ell_1\}}[\Sigma^{\ell_0+i}X_0,\Sigma^{\ell_1+i}X_1]^{\partial\bbI^{d+i}}.\]

There are obvious functors $J_k\colon\mathsf{Ho}(\partial\bbI^k/\mcM^\mathrm{fin})\to\mathsf{Sp}_\mcM$ given by $X\mapsto\Sigma^0X$. We have the following diagram that commutes up to a natural isomorphism
\[\xymatrix@C=3pc{
\mathsf{Ho}(\partial\bbI^k/\mcM^\mathrm{fin}) \ar[r]^-{J_k} \ar[d]_-{\mathbb L\Sigma} & \mathsf{Sp}_\mcM \ar[d]^-\Sigma \\
\mathsf{Ho}(\partial\bbI^{k+1}/\mcM^\mathrm{fin}) \ar[r]_-{J_{k+1}} & \mathsf{Sp}_\mcM
}\]
where $\mathbb L\Sigma$ denotes the total left derived functor of $\Sigma$ and where the suspension functor on the right is $\Sigma^\ell X\mapsto\Sigma^{\ell+1}X$; it is clearly an equivalence onto its image. Thus, the suspension functor in $\mcM^\mathrm{fin}$ is turned into an equivalence in $\mathsf{Sp}_\mcM$.

\vskip 20pt
\vfill
\vbox{\footnotesize%
\noindent\begin{minipage}[t]{0.5\textwidth}
{\scshape
Luk\'a\v{s} Vok\v{r}\'inek}
\vskip 2pt
Department of Mathematics and Statistics,\\
Masaryk University,\\
Kotl\'a\v{r}sk\'a~2, 611~37~Brno,\\
Czech Republic
\vskip 2pt
\url{koren@math.muni.cz}
\end{minipage}
}

\end{document}